\documentclass[11pt,reqno]{amsart} 
\usepackage[a4paper, margin=3cm]{geometry}

\usepackage{graphicx}
\graphicspath{ {./images/} }
\usepackage{lipsum}
\usepackage{subcaption}

\usepackage{xcolor}

\usepackage{amsthm}
\usepackage{bbm}
\usepackage{mathrsfs}
\usepackage{dynkin-diagrams}

\usepackage[utf8]{inputenc}

\usepackage{tikz-cd}

\usepackage{enumitem}

\usepackage{amssymb}
\usepackage{float}
\usepackage{amsbsy}
\usepackage[all]{xy}\usepackage{xypic}
\usepackage{braket}
\usepackage[colorlinks = true,citecolor=black]{hyperref}
\usepackage{amsmath}

\usepackage{wrapfig}

\usepackage[upint]{stix}

\hypersetup{
     colorlinks=true,
     linkcolor=blue,
     filecolor=blue,
     citecolor = blue,      
     urlcolor=cyan,
     }

\makeatletter
\newcommand\opteq[1]{\mathrel{\mathpalette\opt@eq{#1}}}
\newcommand{\opt@eq}[2]{%
  \begingroup
  \sbox\z@{$#1#2$}%
  \sbox\tw@{\resizebox{!}{.5\ht\z@}{$\m@th#1($}}%
  \nonscript\hskip-\wd\tw@
  \mkern1mu
  \raisebox{-.35\ht\z@}[0pt][0pt]{\resizebox{!}{.5\ht\z@}{$\m@th#1($}}%
  \mkern-1mu
  {#2}%
  \mkern-1mu
  \raisebox{-.35\ht\z@}[0pt][0pt]{\resizebox{!}{.5\ht\z@}{$\m@th#1)$}}%
  \mkern1mu
  \nonscript\hskip-\wd\tw@
  \endgroup
}
\makeatother

\def\XXint#1#2#3{{\setbox0=\hbox{$#1{#2#3}{\int}$}
     \vcenter{\hbox{$#2#3$}}\kern-.5\wd0}}

\newtheorem{theorem}{Theorem}[section]
\newtheorem*{theorem*}{Theorem}
\newtheorem{theorem-non}{Theorem}
\newtheorem{lemma-non}{Lemma}

\theoremstyle{definition} 
\newtheorem{thm}{Theorem}

\theoremstyle{definition} 
\newtheorem{corollarynon}{Corollary}

\newtheorem{conjecture-non}{Conjecture}

\newtheorem{corollary-non}{Corollary}
\newtheorem{proposition}[theorem]{Proposition}

\newtheorem*{lemma*}{Lemma}
\newtheorem{corollary}[theorem]{Corollary}

\newtheorem*{conjecture*}{Conjecture}

\theoremstyle{definition}
\newtheorem{definition}[theorem]{Definition}
\newtheorem{example}[theorem]{Example}

\theoremstyle{remark}
\newtheorem{remark}[theorem]{Remark}

\numberwithin{equation}{section}

\begin{document}
\title{Bundle type sub-Riemannian structures on holonomy bundles}

\author[E. Correa, G. Galindo, L. Grama]{Eder M. Correa, \ \ Giovane Galindo, \ \ Lino Grama}

\address{{IMECC-Unicamp, Departamento de Matem\'{a}tica. Rua S\'{e}rgio Buarque de Holanda 651, Cidade Universit\'{a}ria Zeferino Vaz. 13083-859, Campinas-SP, Brazil}}
\address{E-mail: {\rm ederc@unicamp.br, giovanepgneto@gmail.com, linograma@gmail.com}}

\maketitle

\begin{abstract} 
In this paper, combining the Rashevsky-Chow-Sussmann (orbit) theorem with the Ambrose-Singer theorem, we introduce the notion of controllable principal connections on principal $G$-bundles. Using this concept, under a mild assumption of compactness, we estimate the Gromov-Hausdorff distance between principal $G$-bundles and certain reductive homogeneous $G$-spaces. In addition, we prove that every reduction of the structure group $G$ to a closed connected subgroup gives rise to a sequence of Riemannian metrics on the total space for which the underlying sequence of metric spaces converges, in the Gromov-Housdorff topology, to a normal reductive homogeneous $G$-space. This last finding allows one to detect the presence of certain reductive homogeneous $G$-spaces in the Gromov-Housdorff closure of the moduli space of Riemannian metrics of the total space of the bundle through topological invariants provided by obstruction theory.
\end{abstract}

\hypersetup{linkcolor=blue}
\tableofcontents

\hypersetup{linkcolor=black}

\section{Introduction} 

Geometric control theory and sub-Riemannian geometry are two areas whose fruitful interaction has been evidenced by the progress in recent decades. Geometric control theory actively employs methods of differential geometry, Lie groups and Lie algebras, and symplectic geometry for studying controllability, motion planning, stabilizability and optimality for nonlinear and linear control systems. On the other hand, sub-Riemannian geometry has been actively adopting methods developed in the scope of the geometric viewpoint of control theory. Application of these methods has led to important results regarding geometry and topology of sub-Riemannian spaces, regularity of sub-Riemannian distances, properties of the group of diffeomorphisms of sub-Riemannian manifolds, local geometry and equivalence of distributions and sub-Riemannian structures, regularity of the Hausdorff volume, etc. For a detailed exposition on this subject, we suggest \cite{jurdjevic1997geometric}, \cite{agrachev2013control}, \cite{montgomery2002tour}, \cite{agrachev2019comprehensive}, \cite{bellaiche1996sub}. A central result which marks a departure point for the interaction between sub-Riemannian geometry and control theory is the Rashevksy-Chow theorem \cite{rashevsky1938connecting}, \cite{chow1940systeme} (generalized by Sussmann to the orbit theorem \cite{sussmann1973orbits}). It asserts that any two points of a connected sub-Riemannian manifold, endowed with a bracket generating distribution, can be joining by a horizontal path. In \cite[Appendix D]{montgomery2002tour}, Montgomery provides a proof for the Ambrose-Singer theorem \cite{ambrose1953theorem} as corollary of the Rashevsky-Chow-Sussmann (orbit) theorem. Based on the ideas introduced in \cite{montgomery2002tour}, this paper aims developing the interplay between geometric control theory and sub-Riemannian geometry. The main purpose is to employ techniques from sub-Riemannian geometry and geometric control theory to study metric geometry on principal bundles, with focus on distance estimates and convergence issues in the Gromov-Housdorff space, i.e., the metric space consisting of isometry classes of compact metric spaces quipped with the Gromov-Hausdorff distance, e.g. \cite{edwards1975structure}, \cite{gromov1999metric}, \cite{brin2001course}, \cite{ivanov2015gromov}. 


\subsection{Main results}In order to state our main results, let us introduce some context and terminology. Given a principal $G$-bundle $\pi \colon P \to M$, such that $M$ is a compact and connected manifold, and $G$ is a compact and connected Lie group, by fixing a principal connection $\theta \in \Omega^{1}(P,\mathfrak{g})$ and a Riemannian metric $g_{P}$ on $P$, such that $(\ker(\theta),g_{P})$ defines a bundle type sub-Riemannian structure on $P$, see for instance \cite[\S 11.2]{montgomery2002tour}, we have a decomposition of $P$ into a disjoint union of accessible sets $P = \bigcup_{u}P_{u}(\theta)$, such that $P_{u}(\theta)$ is the holonomy bundle of $\theta$ through $u \in P$. From the reduction theorem (e.g. \cite{kobayashi1963foundations}), for every $u \in P$, we have that $P_{u}(\theta) \to M$ is a principal ${\rm{Hol}}_{u}(\theta)$-bundle and the restriction of $\theta$ to $P_{u}(\theta)$ defines a principal connection. Considering the induced sub-Riemannian structure 
\begin{equation}
(\mathcal{H} := \ker(\theta)|_{P_{u}(\theta)},g_{P}|_{\mathcal{H}}),
\end{equation}
on $P_{u}(\theta)$, it can be shown from the Ambrose-Singer theorem that  $\mathcal{H} \subset T(P(\theta))$ defines a bracket generating distribution (e.g. Theorem \ref{controllableconnection}). From this, considering the Carnot-Carath\'{e}odory distance induced by $(\mathcal{H},g_{P}|_{\mathcal{H}})$ on $P_{u}(\theta)$, i.e.,
\begin{equation}
{\rm{dist}}_{\mathcal{H}}(a,b) := \inf \Big \{g_{P}{\text{-lengths of}} \ \mathcal{H}{\text{-horizontal curves between}} \ a \ \text{and} \ b \Big \}, 
\end{equation}
for every $a,b \in P_{u}(\theta)$, see for instance \cite{gromov1996carnot}, it follows from the Rashevsky-Chow theorem (e.g. \cite{chow1940systeme}, \cite{rashevsky1938connecting}) that $(P_{u}(\theta), {\rm{dist}}_{\mathcal{H}})$ is a metric space whose topology induced by ${\rm{dist}}_{\mathcal{H}}$ coincides with the underlying manifold topology. In particular, fixed a bi-invariant metric $\langle -,- \rangle$ on $G$, a Riemannian metric $g_{M}$ on $M$, let us consider 
\begin{equation}
g_{t} := f(t)\pi^{\ast}g_{M} + \langle \theta,\theta \rangle  \ \ \text{and} \ \ g_{P} := \pi^{\ast}g_{M} + \langle \theta,\theta \rangle,
\end{equation}
for some continuous function $f \colon [0,T) \to \mathbbm{R}$, such that $T > 0$ and $f > 0$. Given a closed connected subgroup $H \subset G$, let us also denote by $g_{T}$ the $G$-invariant Riemannian metric induced by $\langle -, - \rangle$ on $G/H$ which makes the projection $\pi \colon G \to G/H$ a Riemannian submersion. Consider now the set of all isometry classes of compact metric spaces $\mathcal{M}$ endowed with the metric $d_{GH}$ defined by the Gromov-Hausdorff distance ( e.g. \cite[Theorem 7.3.25]{brin2001course}, \cite[Theorem 2.1]{kalton1999distances})
\begin{equation}
d_{GH}\big ((X,d_{X}),(Y,d_{Y}) \big ) = \frac{1}{2}\inf \Big  \{ {\rm{dis}}(\mathscr{R}) \ \ \Big | \ \ \mathscr{R}  \in \mathcal{R}(X,Y) \Big\},
\end{equation}
for all $(X,d_{X}),(Y,d_{Y}) \in \mathcal{M}$, here $\mathcal{R}(X,Y)$ denotes the set of all correspondences between $X$ and $Y$, and  ${\rm{dis}}(\mathscr{R})$ denotes the distortion of $\mathscr{R} \in \mathcal{R}(X,Y)$, namely, 
\begin{equation}
{\rm{dis}}(\mathscr{R}) = \sup \Big  \{ |d_{X}(x,x') - d_{Y}(y,y')| \ \ \big | \ \ (x,y), (x',y') \in \mathscr{R} \ \Big\}.
\end{equation}
For more details on this subject, we refer \cite{brin2001course}. In the above context, we prove the following.
\begin{thm}
\label{thmA}
If $\theta \in \Omega^{1}(P,\mathfrak{g})$ is a principal connection, such that ${\rm{Hol}}_{u}(\theta) \subset G$ is a closed subgroup, then 
\begin{equation}
d_{GH}\big ((P,d_{g_{t}}),(G/{\rm{Hol}}_{u}(\theta),d_{g_{T}})\big) \leq \frac{{\rm{diam}}\big (P_{u}(\theta),{\rm{dist}}_{\mathcal{H}} \big )f(t)}{2},
\end{equation}
$\forall t \in [0,T)$, where $P_{u}(\theta)$ is the holonomy bundle through $u \in P$ and ${\rm{dist}}_{\mathcal{H}}$ is the Carnot-Carath\'{e}odory metric induced by the bundle type sub-Riemannnian structure $(\ker(\theta),g_{P})$ restricted to $P_{u}(\theta)$.
\end{thm}
In the setting of the above theorem, for every $t \in [0,T)$, we construct a proper surjective Riemannian submersion with connected fibers
\begin{equation}
F \colon (P,g_{t}) \to (G/{\rm{Hol}}_{u}(\theta),g_{T}),
\end{equation}
which defines a Serre fibration (by Ehresmann's theorem), thus we have the long exact sequence of homotopy groups (e.g. \cite[Theorem 6.3.2]{tom2008algebraic})
\begin{center}
\begin{tikzcd} 
\cdots \arrow[r] & \pi_{j}(P_{u}(\theta)) \arrow[r] &  \pi_{j}(P) \arrow[r] & \pi_{j}(G/{\rm{Hol}}_{u}(\theta)) \arrow[r]  & \pi_{j-1}(P_{u}(\theta)) \arrow[r] & \cdots \end{tikzcd}
\end{center}
In particular, since $\pi_{0}(P_{u}(\theta)) = 0$, if $\pi_{1}(P)$ satisfies a property $\mathscr{P}$ of groups which is preserved under surjective group homomorphisms, for instance, being trivial, or finite, cyclic, abelian, solvable, without subgroups of index 2, etc., then $\pi_{1}(G/{\rm{Hol}}_{u}(\theta))$ has this property, too (cf. \cite{Tuschmann1995}). In general, if we have a sequence of compact length spaces $\{(X_{n},d_{n})\}$, such that 
\begin{equation}
(X_{n},d_{n}) \overset{G.H.}{\longrightarrow} (X_{\infty},d_{\infty}), 
\end{equation}
as $n \uparrow +\infty$, then for $n \gg 0$ we have a surjective group homomorphism between the fundamental groups of $X_{n}$ and $X_{\infty}$. However, not much can be said beyond this. Despite this, our next result shows that Theorem \ref{thmA} provides a constructive framework to build convergent sequences in the Gromov-Hausdorff space for which the homotopy type of the associated limit can be described through some initial data. More precisely, we have the following.

\begin{corollarynon}
\label{corollaryA}
Under the hypotheses of the last theorem, if $\lim_{t \to T}f(t) = 0$, then
\begin{equation}
\lim_{t \to T} d_{GH}\big ((P,d_{g_{t}}),(G/{\rm{Hol}}_{u}(\theta),d_{g_{T}})\big)  = 0.
\end{equation}
\end{corollarynon}
As it can be seen, if $\lim_{t \to T}f(t) = 0$, by taking $0 \leq t_{n} < T$, such that $t_{n} \to T$, as $n \uparrow +\infty$, we have a sequence of compact metric spaces $(X_{n},d_{n})$, such that $X_{n} = P$ and $d_{n} = d_{g_{t_{n}}}$, for all $n \in \mathbbm{N}$, which converges to $(X_{\infty},d_{\infty})$, such that $X_{\infty} = G/{\rm{Hol}}_{u}(\theta)$ and $d_{\infty} = d_{g_{T}}$. Moreover, we obtain from the previous comments the following long exact sequence of homotopy groups
\begin{center}
\begin{tikzcd} 
\cdots \arrow[r] & \pi_{j}(P_{u}(\theta)) \arrow[r] &  \pi_{j}(X_{n}) \arrow[r] & \pi_{j}(X_{\infty}) \arrow[r]  & \pi_{j-1}(P_{u}(\theta)) \arrow[r] & \cdots \end{tikzcd}
\end{center}
for all $n \in \mathbbm{N}$. Hence, the higher homotopy groups of $X_{\infty}$ can be determined by the homotopy groups of $P$ and $P_{u}(\theta)$. In particular, if $P_{u}(\theta)$ is $k$-connected, then 
\begin{equation}
\pi_{j}(X_{n}) \cong \pi_{j}(X_{\infty}), \ \ \forall j \leq k, \ \ \ \text{and} \ \ \ H_{k+1}(X_{n}) \cong H_{k+1}(X_{\infty}),
\end{equation}
see for instance \cite[Theorem 20.1.1]{tom2008algebraic}. From above, we obtain topological constraints on the limit $X_{\infty}$ imposed by the topology of $P$ and $P_{u}(\theta)$. Also, we observe that the homotopy groups of $P$ and $P_{u}(\theta)$ are subjected to constraints provided by the long exact sequence of homotopy groups
\begin{center}
\begin{tikzcd} 
\cdots \arrow[r] & \pi_{j}(G) \arrow[r] &  \pi_{j}(P) \arrow[r] & \pi_{j}(M) \arrow[r]  & \pi_{j-1}(G) \arrow[r] & \cdots \end{tikzcd}
\end{center}
\begin{center}
\begin{tikzcd} 
\cdots \arrow[r] & \pi_{j}({\rm{Hol}}_{u}(\theta)) \arrow[r] &  \pi_{j}(P_{u}(\theta)) \arrow[r] & \pi_{j}(M) \arrow[r]  & \pi_{j-1}({\rm{Hol}}_{u}(\theta)) \arrow[r] & \cdots \end{tikzcd}
\end{center}
Following \cite{nomizu1955reduction, nomizu1956theoreme} and using the above diagrams, one can set some criteria to decide whether or not $G$ can be reduced to a closed connected subgroup $H \subset G$ and, consequently, whether or not $G/H$ can be realized as the Gromov-Hausdorff limit of a sequence of compact metric spaces, as guaranteed by Corollary \ref{corollaryA}. It is worth pointing out that the results provided by Theorem \ref{thmA} and Corollary \ref{corollaryA} generalize some techniques employed in \cite[Appendix A]{correa2023levi} and \cite[\S 4]{tosatti2013chern} to describe the Gromov-Hausdorff limit of the Chern-Ricci flow on certain principal $T^{2}$-bundles over K\"{a}hler-Einstein Fano manifolds. In this direction, Theorem \ref{thmA} allows one to obtain estimates for the limiting behavior of the lift of geometric flows of basic metrics through the following corollaries.
\begin{corollarynon}
Under the hypotheses of the last theorem, 
\begin{equation}
d_{GH}\big ((P,d_{g_{P}}),(G/{\rm{Hol}}_{u}(\theta),d_{g_{T}})\big) \leq \frac{{\rm{diam}}\big (P_{u}(\theta),{\rm{dist}}_{\mathcal{H}} \big )}{2},
\end{equation}
where $P_{u}(\theta)$ is the holonomy bundle through $u \in P$ and ${\rm{dist}}_{\mathcal{H}}$ is the Carnot-Carath\'{e}odory metric induced by the bundle type sub-Riemannnian structure $(\ker(\theta),g_{P})$ restricted to $P_{u}(\theta)$.
\end{corollarynon}

\begin{corollarynon}
In the setting of the last corollary, if $g_{P}(t) := \pi^{\ast}g_{M}(t) + \langle \theta,\theta \rangle$, $t \in [0,T)$, for some family of metrics $\{g_{M}(t)\}_{t \in [0,T)}$ on the base manifold $M$, then
\begin{equation}
d_{GH}\big ((P,d_{g_{P}(t)}),(G/{\rm{Hol}}_{u}(\theta),d_{g_{T}})\big) \leq \frac{{\rm{diam}}\big (P_{u}(\theta),{\rm{dist}}_{\mathcal{H}}(t) \big )}{2},
\end{equation}
where $P_{u}(\theta)$ is the holonomy bundle through $u \in P$ and ${\rm{dist}}_{\mathcal{H}}(t)$ is the Carnot-Carath\'{e}odory metric induced by the bundle type sub-Riemannnian structure $(\ker(\theta),g_{P}(t))$ restricted to $P_{u}(\theta)$.
\end{corollarynon}

From the above result we see that the convergence of $(P,d_{g_{P}(t)})$ to $(G/{\rm{Hol}}_{u}(\theta),d_{g_{T}})$ in the Gromov-Hausdorff space is controlled by ${\rm{diam}} (P_{u}(\theta),{\rm{dist}}_{\mathcal{H}}(t))$, $t \in [0,T)$, in the sense that
\begin{equation}
{\rm{diam}} (P_{u}(\theta),{\rm{dist}}_{\mathcal{H}}(t)) \rightarrow 0, \ \text{as} \ t \to T \ \Rightarrow (P,d_{g_{P}(t)}) \overset{G.H.}{ \longrightarrow} (G/{\rm{Hol}}_{u}(\theta),d_{g_{T}}), \ \text{as} \ t \to T.
\end{equation}
Given an arbitrary closed connected subgroup $H \subset G$, our next result relates the existence of $H$-reductions of $P$ with the existence of convergent sequences in the Gromov-Hausdorff space. More precisely, we establish the following result.
\begin{thm}
\label{thmB}
Let $\pi \colon P \to M$ be a principal $G$-bundle, such that $M$ and $G$ are both compact and connected, and $\dim(M) \geq 2$. Consider the subset $\mathcal{M}(P) \subset (\mathcal{M},d_{GH})$ defined by 
\begin{equation}
\mathcal{M}(P) := \big \{ (P,d_{g}) \ | \ d_{g} \ \text{is the distance induced by} \ g \in {\rm{Sym}}_{+}^{2}(T^{\ast}P) \big \}.
\end{equation}
If $\pi \colon P \to M$ is reducible to a closed connected subgroup $H \subset G$, then $(G/H,d_{g_{T}}) \in \overline{\mathcal{M}(P)}^{GH}$.
\end{thm}

In the above setting, if $(M,g)$ is a compact orientable Riemannian manifold, and $P = P_{{\rm{SO}}(n)}(M)$ (orthonormal frame bundle), it follows that the existence of a $G$-structure on $M$ \cite{chern1966geometry}, for some closed connected Lie subgroup $G \subset {\rm{SO}}(n)$, implies that 
\begin{equation}
({\rm{SO}}(n)/G,d_{g_{T}}) \in \overline{\mathcal{M}(P_{{\rm{SO}}(n)}(M))}^{GH}. 
\end{equation}
Thus, the presence of certain homogeneous manifolds in the GH-closure $\overline{\mathcal{M}(P_{{\rm{SO}}(n)}(M))}^{GH}$ can be detected by topological invariants (e.g. characteristic classes). Further, consider the moduli space of Riemannian structures on $P$, also known as superspace \cite{edwards1975structure}, i.e.,
\begin{equation}
\mathcal{S}(P):= {\mathcal{R}}(P)/{\rm{Diff}}(P),
\end{equation}
such that ${\mathcal{R}}(P) = {\rm{Sym}}_{+}^{2}(T^{\ast}P)$, see for instance \cite{ebin1968space}, \cite{fischer1970theory}, \cite{tuschmann2015moduli}. Under the identification $\mathcal{S}(P) \simeq \mathcal{M}(P)$, provided by $[g] \mapsto (P,d_{g})$, we can realize $\mathcal{S}(P) \subset (\mathcal{M},d_{GH})$. Thus, it follows from Theorem \ref{thmB} that the presence of certain reductive homogeneous spaces in $\overline{\mathcal{S}(P)}^{GH}$ can be detected through topological invariants provided by obstruction theory. Just to illustrate, given a compact orientable 6-manifold $M$, if $w_{3}(M) = 0$, it follows from Theorem \ref{thmB} that 
\begin{equation}
 (\mathbbm{P}^{3} = {\rm{SO}}(6)/{\rm{U}}(3), d_{g_{{\rm{FS}}}}) \in \overline{\mathcal{S}(P_{{\rm{SO}}(6)}(M))}^{GH},
\end{equation}
where $w_{3}(M) \in H^{3}(M,\mathbbm{Z}_{2})$ is the third Stiefel–Whitney class of $M$, e.g. \cite[Theorem 9]{wall1966classification} and \cite{massey1961obstructions}. In particular, since $P_{{\rm{SO}}(6)}(S^{6}) = {\rm{SO}}(7)$, we have $(\mathbbm{P}^{3}, d_{g_{{\rm{FS}}}}) \in \overline{\mathcal{S}( {\rm{SO}}(7))}^{GH}$.
\subsection*{Acknowledgments.}E. M. Correa is supported by S\~{a}o Paulo Research Foundation FAPESP grant 2022/10429-3. L. Grama research is partially supported by S\~ao Paulo Research Foundation FAPESP grants 2018/13481-0 and 2023/13131-8.

\section{Generalities on sub-Riemannian geometry}

Throughout this paper, unless otherwise stated, every smooth manifold is connected. In this section, we shall present some classical results on sub-Riemannian geometry, for more details on this subject, we refer \cite{montgomery2002tour}, \cite{agrachev2019comprehensive}, \cite{chow1940systeme}, \cite{rashevsky1938connecting}.
\begin{definition}
A sub-Riemannian structure on a smooth manifold $Q$ is a pair $(\mathcal{H},\langle -,-\rangle)$, such that $\mathcal{H} \subset TQ$ is a subbundle and $\langle -,-\rangle$ is a fiber inner product on $\mathcal{H}$.
\end{definition}

Given a subbundle $\mathcal{H} \subset TQ$, for the sake if simplicity, we shall also denote by $\mathcal{H}$ the distribution induced by $\mathcal{H}$ on $Q$. In this setting, we have the following definition.

\begin{definition}
A subbundle $\mathcal{H} \subset TQ$ is bracket-generating if 
\begin{equation}
 TQ = {\rm{Lie}}(\mathcal{H}) := {\rm{Span}} \Big \{ [X_{1}, \cdots, [X_{j-1},X_{j}]] \ \ \Big | \ \ X_{j} \in \mathcal{H}, \ \ j \in \mathbbm{N}\Big \}.
\end{equation}
\end{definition}

\begin{definition}
Let $(Q,\mathcal{H},\langle -,-\rangle)$ be a sub-Riemannian manifold. Given $a,b \in Q$, if there exists a $\mathcal{H}$-horizontal curve $\gamma \colon [0,1] \to Q$, such that $\gamma(0) = a$ and $\gamma(1) = b$, then the sub-Riemannian distance (a.k.a. Carnot-Carath\'{e}odory distance) between $a$ and $b$ is defined by 
\begin{equation}
{\rm{dist}}_{\mathcal{H}}(a,b) = \inf \Bigg \{ \int_{0}^{1}||\dot{\gamma}(s)||{\rm{d}}s \ \ \big | \ \ \gamma \ \text{is a} \ \mathcal{H}\text{-horizontal curve joining} \ a \ {\text{and}} \ b\Bigg \}.
\end{equation}
If there is no $\mathcal{H}$-horizontal curve joining $a$ and $b$, we set $d_{\mathcal{H}}(a,b) = +\infty$.
\end{definition}

\begin{theorem}[Chow]
If a distribution $\mathcal{H} \subset TQ$ is bracket-generating, then every point of $Q$ can be connected by a $\mathcal{H}$-horizontal curve.
\end{theorem}

\begin{theorem}[Rashevsky-Chow]
Let $(Q,\mathcal{H},\langle -,-\rangle)$ be a sub-Riemannian manifold, such that $\mathcal{H} \subset TQ$ is bracket-generating, then:
\begin{enumerate}
\item[(1)] $(Q,{\rm{dist}}_{\mathcal{H}})$ is a metric space,
\item[(2)] the topology induced by $(Q,d_{\mathcal{H}})$ is equivalent to the manifold topology.
\end{enumerate}
In particular, $d_{\mathcal{H}} \colon Q\times Q \to \mathbbm{R}$ is continuous.
\end{theorem}

\section{Geometric control theory and principal connections}
In this section, we shall present some basic results related with geometric control theory, with focus on the orbit theorem \cite{jurdjevic1997geometric}, \cite{agrachev2013control}, \cite{sussmann1973orbits}. After that, we recall some classical results and concepts on the theory of principal connections and holonomy, the main purpose is to introduce the notion of controllable connections. The subject will be presented according to the following references \cite{rudolph2017differential}, \cite{kolar2013natural}, \cite{kobayashi1963foundations}, \cite{ambrose1953theorem}.
\begin{definition}
A control system on a smooth manifold $Q$ is defined by an arbitrary family $\mathcal{F}$ of smooth vector fields on $Q$.
\end{definition}

Given an arbitrary family $\mathcal{F}$ of smooth vector fields on $Q$, let us denote by $G(\mathcal{F})$ the pseudogroup of (local) diffeomorphisms generated by the flows of elements of $\mathcal{F}$, i.e., if $\Phi \in G(\mathcal{F})$, then
\begin{equation}
\Phi = \Phi_{t_{1}}^{X_{1}} \circ \cdots \circ \Phi_{t_{k}}^{X_{k}},
\end{equation}
such that $X_{j} \in \mathcal{F}$, for all $j = 1, \ldots,k$, $k \in \mathbbm{N}$, and $t_{j} \in \mathbbm{R}$ are chosen so that the composition make sense. In this last setting, given $q_{0} \in Q$, let us denote by $\mathcal{O}(q_{0})$ the orbit of $G(\mathcal{F})$ through $q_{0}$.

\begin{theorem}[Sussmann] Let $\mathcal{F}$ be an arbitrary family of smooth vector fields on $Q$ and $q_{0} \in Q$. Then:
\begin{enumerate}
\item[(1)] $\mathcal{O}(q_{0})$ is an immersed submanifold of $Q$,
\item[(2)] $T_{x}(\mathcal{O}(q_{0})) = {\rm{Span}} \big \{ (\Phi_{\ast}X)(x) \ | \ \Phi \in G(\mathcal{F}), \ X \in \mathcal{F}\}$.
\end{enumerate}
\end{theorem}
From the above theorem, it can be shown that 
\begin{equation}
{\rm{Lie}}(\mathcal{F})_{x} \subseteq  T_{x}(\mathcal{O}(q_{0})),
\end{equation}
for every $x \in \mathcal{O}(q_{0})$. For us it will be important the following result.
\begin{theorem}[Rashevsky-Chow] 
\label{RC-orbit}
If $\mathcal{F}$ is bracket-generating, then $\mathcal{O}(q) = Q$, for every $q \in Q$.
\end{theorem}

Let us recall some standard results related with the theory of Cartan-Ehresman connections on principal bundles. Let $M$ be a smooth manifold and $G$ a Lie group. A smooth fiber bundle $\pi \colon P \to M$ with fiber $G$ is a principal $G$-bundle if $G$ acts smoothly and freely on $P$ on the right and the fiber-preserving local trivializations
\begin{equation}
\phi_{U} \colon \pi^{-1}(U) \to U \times G,
\end{equation}
are $G$-equivariant, where $G$ acts on $U \times G$ on the right by $(x,h) \cdot g = (x,hg)$. Denoting by $\mathfrak{g}$ the Lie algebra of $G$, we have the following definition.
\begin{definition}
A Cartan-Ehresman connection on a principal $G$-bundle $G \hookrightarrow P \to M$ is defined by a $\mathfrak{g}$-valued $1$-form $\theta \colon TP \to \mathfrak{g}$ on $P$, satisfying the following
\begin{equation}
\label{connection}
R_{g}^{\ast}\theta = {\rm{Ad}}(g^{-1})\theta, \ \forall g \in G, \ \ \   \theta(\xi_{\ast}) = \xi, \ \forall \xi \in \mathfrak{g},
\end{equation}
where $\xi_{\ast}$ is the Killing vector field generated by $\xi \in \mathfrak{g}$, and $R_{g}(x) = xg$, $\forall x \in P$ and $\forall g \in G$.
\end{definition}

In the above setting, we have the decomposition 
\begin{equation}
TP = H(P) \oplus V(P),
\end{equation}
such that $H(P) = \ker(\theta)$ and, for every $x \in P$, $V(P)_{x} = j_{x \ast}(\mathfrak{g})$, where $j_{x} \colon G \to P$ is the map defined by $j_{x}(g) = xg$, $\forall g \in G$. Moreover, from Eq. (\ref{connection}), we have
\begin{equation}
H(P)_{xg} = R_{g\ast}H(P)_{x},
\end{equation}
$\forall x \in P$ and $\forall g \in G$. In this last setting, a curve $\gamma \colon I \to P$ is said to be $\theta$-horizontal if 
\begin{equation}
\dot{\gamma}(t) \in H(P)_{\gamma(t)}, \ \ \forall t \in I.
\end{equation}
Given a principal $G$-bundle $G \hookrightarrow P \to M$, fixed a principal connection $\theta \in \Omega^{1}(P,\mathfrak{g})$, we have the following well-known results.

\begin{proposition}
Let $\alpha \colon [0,1] \to M$ be a smooth curve and let $X \colon U \to TM$ be a smooth vector field, such that $U\subseteq M$ is an open set. Then, the following holds:
\begin{enumerate}
\item[(i)] For every $x \in \pi^{-1}(\alpha(0))$, there exists a unique smooth curve $\alpha_{x}^{h} \colon [0,1] \to P$, such that 
\begin{equation}
\alpha_{x}^{h}(0) = x, \ \ \ \dot{\alpha}_{x}^{h}(t) \in H(P)_{\alpha_{x}^{h}(t)}, \ \ \text{and} \ \ \pi(\alpha_{x}^{h}(t)) = \alpha(t), \ \ \forall t \in [0,1];
\end{equation}
\item[(ii)] There exists a unique smooth vector field $X^{H} \colon \pi^{-1}(U) \to TP$, such that 
\begin{equation}
X^{H}(x) \in H(P)_{x}, \ \ \text{and} \ \ \pi_{\ast}(X^{H}(x))=X(x), \forall x \in \pi^{-1}(U).
\end{equation}
\end{enumerate}
\end{proposition}

From item (i) of the above proposition, we obtain the following result.
 
\begin{proposition}
Let $\alpha \colon [0,1] \to M$ be a piecewise smooth curve. Then, the map
\begin{equation}
\tau_{\alpha} \colon \pi^{-1}(\alpha(0)) \to \pi^{-1}(\alpha(1)), \ \tau_{\alpha}(x)= \alpha_{x}^{h}(1), \forall x\in \pi^{-1}(\alpha(0)),
\end{equation}
is a $G$-equivariant diffeomorphism.
\end{proposition}
\begin{definition}
$\tau_{\alpha}$ is called the parallel transport along $\alpha$ with respect to the connection $\theta$.
\end{definition}
Now we define an equivalence relation $\sim$ on $P$ by saying that $p \sim q$ if, and only if, they can be joined by a piecewise smooth $\theta$-horizontal path in $P$. If $\alpha$ is a loop based at $m \in M$, then, $\forall x \in \pi^{-1}(m)$, we have $x \sim \tau_{\alpha}(x) = xg$, for some $g \in G$. 
\begin{definition}
The holonomy group of $\theta$ based at $u \in P$ is defined by
\begin{equation}
{\rm{Hol}}_{u}(\theta) = \big \{ g \in G \ \  | \ \ u \sim ug\big \}.
\end{equation}
The restricted holonomy group of $\theta$ based at $u$ is defined by the subgroup ${\rm{Hol}}_{u}^{0}(\theta)$ of parallel transports along contractible loops based at $\pi(u) \in M$.
\end{definition}

The holonomy group ${\rm{Hol}}_{u}(\theta) \subset G$, $u \in P$, is a Lie subgroup\footnote{ Recall that $H \subset G$ is a Lie subgroup if $H$ is an immersed submanifold of $G$, such that the product $H \times H \to H$ is differentiable with respect to the intrinsic structure of $H$.} of $G$ whose the connected component of the identity is given by the connected Lie subgroup ${\rm{Hol}}_{u}^{0}(\theta) \subset G$, which is called reduced holonomy group of $\theta$. Moreover, we have the following properties:
\begin{enumerate}
\item[(1)] ${\rm{Hol}}_{ug}(\theta) = g^{-1}{\rm{Hol}}_{u}(\theta)g$, $\forall g \in G$,
\item[(2)] If $M$ is simply connected, then ${\rm{Hol}}_{u}(\theta) = {\rm{Hol}}_{u}^{0}(\theta)$, $\forall u \in P$.
\end{enumerate}
\begin{remark}
Given $u,v \in P$, since ${\rm{Hol}}_{v}(\theta) = g^{-1}{\rm{Hol}}_{u}(\theta)g$, for some $g \in G$, we can also denote the holonomy group of $\theta$ by ${\rm{Hol}}(\theta)$ without referring to the base point.
\end{remark}

\begin{definition}
A reduction of a principal $G$-bundle $\pi \colon P \to M$ to a Lie subgroup $H$ is defined by a principal $H$-bundle $\pi_{Q} \colon Q \to M$ together with a $H$-equivariant smooth map $\Psi \colon Q \to P$, which covers the identity map ${\rm{id}}_{M} \colon M \to M$.
\end{definition}

Given a principal $G$-bundle $\pi \colon P \to M$, fixed a connection $\theta \in \Omega^{1}(P,\mathfrak{g})$, for every $u \in P$, let us consider the following subset 
\begin{equation}
\label{Holonomybundle}
P_{u}(\theta): = \big \{ x \in P \ \ | \ \ u \sim x\big \},
\end{equation}
i.e., the set of points which can be joined to $u$ by a $\theta$-horizontal path. 
\begin{theorem}
\label{holonomybundle}
In the last setting, the following hold:
\begin{enumerate}
\item[(a)] $P_{u}(\theta)$ is a principal ${\rm{Hol}}_{u}(\theta)$-bundle over $M$;
\item[(b)] The restriction of $\theta$ to $P_{u}(\theta)$ defines a principal connection.
\end{enumerate}
\end{theorem}
\begin{remark}
Notice that $P_{u}(\theta) \hookrightarrow P$ is a reduction of $\pi \colon P \to M$ to ${\rm{Hol}}_{u}(\theta) \subset G$. The principal bundle $P_{u}(\theta)$ is called holonomy bundle of $\theta$ through $u \in P$.
\end{remark}
\begin{theorem}[Ambrose-Singer]
\label{AS-theo}
Let $P$ be a principal $G$-bundle over $M$, and let $\theta$ be a principal connection on $P$. Then, for every $u \in P$, the Lie algebra $\mathfrak{hol}_{u}(\theta)$ of ${\rm{Hol}}_{u}(\theta)$ is given by
\begin{equation}
{\mathfrak{hol}}_{u}(\theta) = {\rm{Span}}\big \{  {\rm{d}}\theta_{x}\big (V,W\big ) \ \  | \ \ x \in P_{u}(\theta), \ \ V,W \in \ker(\theta)_{x}\big\}.
\end{equation}
\end{theorem}

Based on the previous ideas related with geometric control theory and principal connections, we introduce the following concept.
\begin{definition}
Given a principal $G$-bundle $\pi \colon P \to M$, a principal connection $\theta \in \Omega^{1}(P;\mathfrak{g})$ is said to be a {\textit{controllable}} if $\mathcal{H} = \ker(\theta) \subset TP$ defines a bracket-generating distribution.
\end{definition}

The next result characterizes completely the controllable principal connections.

\begin{theorem}
\label{controllableconnection}
Let $P$ be a principal $G$-bundle over $M$, and let $\theta$ be a principal connection on $P$. Then, $\theta$ is controllable if, and only if, $G =  {\rm{Hol}}_{u}(\theta)$ and $P_{u}(\theta) = P$, for every $u \in P$.
\end{theorem}

\begin{proof}
Denoting $\mathcal{H} = \ker(\theta)$, if $\theta$ is controllable, by definition, we have ${\rm{Lie}}(\mathcal{H}) = TP$. Thus, given $u \in P$, from Theorem \ref{RC-orbit}, we have that $\mathcal{O}(q) = \{\Phi(u) \ | \ \Phi \in G(\mathcal{H})\} = P$. Since $\mathcal{O}(q) \subset P_{u}(\theta)$, we conclude that $P = P_{u}(\theta)$. Therefore, $u \sim ug$, for every $g \in G$, i.e., $G = {\rm{Hol}}_{u}(\theta)$. Conversely, suppose that $G = {\rm{Hol}}_{u}(\theta)$ and $P_{u}(\theta) = P$, for every $u \in P$. In this case, from Theorem \ref{holonomybundle} and Theorem \ref{AS-theo}, for every $u,v \in P$, we have 
\begin{equation}
T_{v}(P_{u}(\theta)) = \mathcal{H}_{v} \oplus j_{v \ast}({\mathfrak{hol}}_{u}(\theta)), 
\end{equation}
such that 
\begin{equation}
{\mathfrak{hol}}_{u}(\theta) = {\rm{Span}}\big \{  {\rm{d}}\theta_{x}\big (V,W\big ) \ \  | \ \ x \in P_{u}(\theta), \ \ V,W \in \ker(\theta)_{x}\big\}.
\end{equation}
Given $x \in P_{u}(\theta)$, since $v \sim x$, we have a $G$-equivariant diffeomorphism $\tau \colon \pi^{-1}(\pi(x)) \to \pi^{-1}(\pi(v))$, such that $v = \tau(x)$, where $\tau$ is the parallel transport along the projection of some $\mathcal{H}$-horizontal curve joining $x$ and $v$. From this, since
\begin{equation}
j_{v}(h) = vh = \tau(x)h = \tau(xh) = \tau(j_{x}(h))
\end{equation}
for every $h \in {\rm{Hol}}_{u}(\theta)$, for all $\mathcal{H}$-horizontal vector fields $X$ and $Y$, we obtain
\begin{equation}
 j_{v \ast} \Big ( \theta_{x}\big ([X,Y](x)\big )\Big ) = \tau_{\ast} \Big ( j_{x \ast} \theta_{x}\big ([X,Y](x) \big )\Big) = \tau_{\ast} \Big ( ([X,Y])^{V}(x)\Big ),
\end{equation}
where $([X,Y])^{V}$ is the vertical component of $[X,Y]$. Since $\tau_{\ast}(Z(x)) = Z(\tau(x))$, for every vertical vector field $Z$, and ${\rm{d}} \theta_{x}(X(x),Y(x))= -\theta_{x} ([X,Y](x))$, we conclude that 
\begin{equation}
 j_{v \ast} \Big ( {\rm{d}} \theta_{x}(X(x),Y(x))\Big ) = -([X,Y])^{V}(v) = -\Big ([X,Y](v) - ([X,Y])^{H}(v)\Big),
\end{equation}
in other words,  $j_{v \ast}({\mathfrak{hol}}_{u}(\theta)) \subset {\rm{Lie}}(\mathcal{H})_{v}$, where
\begin{equation}
 {\rm{Lie}}(\mathcal{H})_{v} = {\rm{Span}} \Big \{ [X_{1}, \cdots, [X_{j-1},X_{j}]](v) \ \ \Big | \ \ X_{j} \in \mathcal{H}, \ \ j \in \mathbbm{N}\Big \}.
\end{equation}
Thus, we conclude that $TP = T(P_{u}(\theta))= {\rm{Lie}}(\mathcal{H})$, i.e., $\theta$ is controllable.
\end{proof}

\begin{remark}
Let $P$ be a principal $G$-bundle over $M$, and let $\theta$ be an arbitrary principal connection on $P$. From the previous results, we have that
\begin{equation}
P = \bigcup_{u \in P}P_{u}(\theta),
\end{equation}
such that the restriction of $\theta$ to each holonomy bundle $P_{u}(\theta)$, $u \in P$, is controllable.
\end{remark}

\begin{remark}
Using Ambrose-Singer theorem, one can show that, for every principal $G$-bundle $P \to M$, with $M$ connected and $\dim(M) \geq 2$, there exists a controllable connection. For the rather technical proof of this fact we refer to \cite{nomizu1956theoreme}.
\end{remark}
\section{Gromov-Hausdorff distance}

In this section, recall the definition of Gromov-Hausdorff distance, for more details on this subject, we suggest \cite[Chapter 7]{brin2001course}. Given two metric spaces $(X,d_{X})$ and $(Y,d_{Y})$, a {\textit{correspondence}} between the underlying sets $X$ and $Y$ is a subset $\mathscr{R} \subseteq  X \times Y$ satisfying the following property: for every $x \in X$ there exists at least one $y \in Y$, such that $(x,y) \in \mathscr{R}$, and similarly for every $y \in Y$ there exists an $x \in X$, such that $(x,y) \in \mathscr{R}$. Let us denote by $\mathcal{R}(X,Y)$ the set of all correspondences between $X$ and $Y$. Now we consider the following definition
\begin{definition}
Let $\mathscr{R} \in \mathcal{R}(X,Y)$ be a correspondence between two metric spaces $(X,d_{X})$ and $(Y,d_{Y})$. The \textit{distortion} of $\mathscr{R}$ is defined by 
\begin{equation}
{\rm{dis}}(\mathscr{R}) = \sup \Big  \{ |d_{X}(x,x') - d_{Y}(y,y')| \ \ \big | \ \ (x,y), (x',y') \in \mathscr{R} \ \Big\}.
\end{equation}
\end{definition}

From above, we can define the Gromov-Hausdorff distance between two metric spaces as follows.

\begin{definition}
We define the Gromov-Hausdorff distance of any two metric spaces $(X,d_{X})$ and $(Y,d_{Y})$ as being 
\begin{equation}
\label{GHdistance}
d_{GH}\big ((X,d_{X}),(Y,d_{Y}) \big ) = \frac{1}{2}\inf \Big  \{ {\rm{dis}}(\mathscr{R}) \ \ \Big | \ \ \mathscr{R}  \in \mathcal{R}(X,Y) \Big\}.
\end{equation}
\end{definition}

\begin{example}
If $F \colon (X,d_{X}) \to (Y,d_{Y})$ is a surjective map, then the graph
\begin{equation}
\mathscr{R}_{F} = \{(x,F(x)) \in X \times Y \ | \ x \in X\},
\end{equation}
defines a correspondence between two metric spaces $(X,d_{X})$ and $(Y,d_{Y})$.
\end{example}
The next result allows us to consider compact metric spaces as points of metric space, more precisely, we have the following theorem (e.g. \cite[Theorem 7.3.30]{brin2001course}). 
\begin{theorem}[Gromov–Hausdorff space] 
Let $\mathcal{M}$ be the set of all equivalence classes of isometric compact metric spaces. Then, $(\mathcal{M},d_{GH})$ is a metric space.
\end{theorem}

\section{Proof of Theorem A}

Let $\pi \colon P \to (M,g_{M})$ be a principal $G$-principal and, such that $(M,g_{M})$ is compact and connected Riemannian manifold and $G$ is compact and connected Lie group. Fixed a principal connection $\theta \colon TP \to \mathfrak{g}$ and fixed a bi-invariant metric $\langle -,- \rangle$ on $G$, we define the following Riemannian metrics on $P$
\begin{equation}
g_{t} := f(t)\pi^{\ast}g_{M} + \langle \theta,\theta \rangle  \ \ \text{and} \ \ g_{P} := \pi^{\ast}g_{M} + \langle \theta,\theta \rangle,
\end{equation}
where $f \colon [0,T) \to \mathbbm{R}$ is a continuous function, such that $T > 0$ and $f > 0$. Given $u \in P$, we consider the Riemannian manifold $( G/{\rm{Hol}}_{u}(\theta),g_{T})$, such that $g_{T}$ is the $G$-invariant Riemannian metric induced by $\langle -, - \rangle$ which makes the canonical projection $\pi \colon G \to G/{\rm{Hol}}_{u}(\theta)$ a Riemannian submersion, here we suppose that ${\rm{Hol}}_{u}(\theta) \subset G$ is a closed subgroup. In this setting, we have the following result (Theorem \ref{thmA}).
\begin{theorem}
\label{holonomy-GH}
If $\theta \in \Omega^{1}(P,\mathfrak{g})$ is a principal connection, such that ${\rm{Hol}}_{u}(\theta) \subset G$ is a closed subgroup, then 
\begin{equation}
d_{GH}\big ((P,d_{g_{t}}),(G/{\rm{Hol}}_{u}(\theta),d_{g_{T}})\big) \leq \frac{{\rm{diam}}\big (P_{u}(\theta),{\rm{dist}}_{\mathcal{H}} \big )f(t)}{2},
\end{equation}
$\forall t \in [0,T)$, where $P_{u}(\theta)$ is the holonomy bundle through $u \in P$ and ${\rm{dist}}_{\mathcal{H}}$ is the Carnot-Carath\'{e}odory metric induced by the bundle type sub-Riemannnian structure $(\ker(\theta),g_{P})$ restricted to $P_{u}(\theta)$.
\end{theorem}
\begin{proof}
Since ${\rm{Hol}}_{u}(\theta) \subset G$ defines a reduction of $P$, we have an equivariant surjective submersion
\begin{equation}
F \colon P \to G/{\rm{Hol}}_{u}(\theta),
\end{equation}
see for instance \cite[Proposition 1.6.2]{rudolph2017differential}. In the above setting, given $v \in P$, we denote
\begin{equation}
P_{v}(\theta) = F^{-1}(F(v)) \subset P.
\end{equation}
We notice that $P_{v}(\theta) \subset P$, $\forall v \in P$, is an embedded compact submanifold defined by the set of points in $P$ which can be connected to $v$ by a $\theta$-horizontal path, see Eq. (\ref{Holonomybundle}). Given $u \in P$, consider the decomposition
\begin{equation}
\mathfrak{g} = \mathfrak{m} \oplus \mathfrak{hol}_{u}(\theta),
\end{equation}
where $\mathfrak{hol}_{u}(\theta) = {\rm{Lie}}({\rm{Hol}}_{u}(\theta))$, $\mathfrak{m} \cong T_{o}(G/{\rm{Hol}}_{u}(\theta))$, and $o = e{\rm{Hol}}_{u}(\theta)$. From this, $\forall X, Y \in \mathfrak{m}$ and $\forall v \in P$, denoting $F(v) = g{\rm{Hol}}_{u}(\theta)$, we obtain
\begin{equation}
\label{submersion}
g_{T}\Big (\frac{d}{dt} \Big|_{t = 0}F(v\exp(tX)),\frac{d}{dt} \Big|_{t = 0}F(v\exp(tY)) \Big ) \Big |_{F(v)} = \langle X(g),Y(g) \rangle_{g} = \langle X, Y\rangle, 
\end{equation}
here we identify $\mathfrak{g}$ with the Lie algebra of right-invariant vector fields on $G$. On the other hand, by definition of $g_{t}$, we have 
\begin{equation}
\label{submersion1}
g_{t}(j_{v \ast}(X),j_{v \ast}(Y)) = g_{t}\Big (\frac{d}{dt} \Big|_{t = 0}v\exp(tX),\frac{d}{dt} \Big|_{t = 0}v\exp(tY) \Big ) \Big |_{v} =  \langle X, Y\rangle.
\end{equation}
Now we observe that  
\begin{equation}
T_{v}P = \ker(\theta)_{v} \oplus j_{v\ast}(\mathfrak{g}) =  \big (\ker(\theta)_{v} \oplus j_{v\ast}(\mathfrak{hol}_{u}(\theta)) \big) \oplus j_{v\ast}(\mathfrak{m}).
\end{equation}
Thus, from Eq. (\ref{submersion}) and Eq. (\ref{submersion1}), it follows that 
\begin{equation}
F \colon (P,g_{t}) \to (G/{\rm{Hol}}_{u}(\theta),g_{T}), 
\end{equation}
is a proper surjective Riemannian submersion. Therefore, given $a,b \in P$, since $(G/{\rm{Hol}}_{u}(\theta),g_{T})$ is a complete Riemannian manifold, there exists a minimizing geodesic $\alpha$ in $G/{\rm{Hol}}_{u}(\theta)$ connecting $F(a)$ and $F(b)$. Considering, respectively, the vertical and horizontal distributions on $P$ given by $\mathscr{V} := \ker(DF)$ and $\mathscr{H} := \ker(DF)^{\perp_{g_{t}}}$, since $\mathscr{H}$ is Ehresmann-complete, there exists a $\mathscr{H}$-horizontal lift $\tilde{\alpha}$ of $\alpha$, such that $\tilde{\alpha}(0) = a$, see for instance \cite[\S 2.1]{pastore2004riemannian}, \cite{besse2007einstein}. Let us denote $z = \tilde{\alpha}(1) \in P_{b}(\theta)$. From this, we have
\begin{equation}
d_{g_{t}}(a,z) \leq \int_{0}^{1}||\dot{\tilde{\alpha}}(s)||_{g_{t}}{\rm{ds}} = \int_{0}^{1}||\dot{\alpha}(s)||_{g_{T}}{\rm{ds}} = d_{g_{T}}(F(a),F(b)),
\end{equation}
so we conclude that 
\begin{equation}
d_{g_{t}}(a,b) \leq d_{g_{t}}(a,z) + d_{g_{t}}(z,b) \leq d_{g_{T}}(F(a),F(b)) +   d_{g_{t}}(z,b).
\end{equation}
In other words, we obtain
\begin{equation}
d_{g_{t}}(a,b) - d_{g_{T}}(F(a),F(b)) \leq d_{g_{t}}(z,b).
\end{equation}
Since $z,b \in P_{b}(\theta)$, we have a $\theta$-horizonal curve $\gamma \colon [0,1] \to P_{b}(\theta)$, such that $\gamma(0)=z$ and $\gamma(1) = b$. Consider now the sub-Riemannian structure of bundle type $(\mathcal{H}:= \ker(\theta)|_{P_{b}(\theta)},g_{P}|_{\mathcal{H}})$ on $P_{b}(\theta)$, see for instance \cite[Definition 11.2.1]{montgomery2002tour}. Since $g_{t}|_{\mathcal{H}} = f(t)(g_{P}|_{\mathcal{H}})$, we have 
\begin{equation}
d_{g_{t}}(z,b) \leq \int_{0}^{1}||\dot{\gamma}(s)||_{g_{t}}{\rm{ds}} = f(t)\int_{0}^{1}||\dot{\gamma}(s)||_{g_{P}}{\rm{ds}}.
\end{equation}
By taking the infimum over all $\mathcal{H}$-horizontal curves $\gamma \colon [0,1] \to P_{b}(\theta)$, such that $\gamma(0)=z$, and $\gamma(1) = b$, we conclude that 
\begin{equation}
d_{g_{t}}(z,b) \leq f(t){\rm{dist}}_{\mathcal{H}}(z,b),
\end{equation}
where ${\rm{dist}}_{\mathcal{H}}$ is the Carnot-Carath\'{e}odory metric on $P_{b}(\theta)$ induced by $(\mathcal{H},g_{P}|_{\mathcal{H}})$. Notice that, from Theorem \ref{controllableconnection}, it follows that the distribution $\mathcal{H}$ is a bracket-generating distribution on $P_{b}(\theta)$, i.e., the restriction of $\theta$ to $P_{b}(\theta)$ is a controllable connection. From Rashevsky-Chow theorem, it follows that the topology induced by Carnot-Carath\'{e}odory metric ${\rm{dist}}_{\mathcal{H}}$ coincides with the manifold topology of $P_{b}(\theta)$. Since $P_{b}(\theta)$ is compact, we have ${\rm{diam}}(P_{b}(\theta),{\rm{dist}}_{\mathcal{H}}) < \infty$. Therefore, we obtain
\begin{equation}
\label{fundinequality}
d_{g_{t}}(z,b) \leq f(t){\rm{dist}}_{\mathcal{H}}(z,b) \leq f(t){\rm{diam}}(P_{b}(\theta),{\rm{dist}}_{\mathcal{H}}).
\end{equation}
Now we observe the following. Given $v_{1},v_{2} \in P$, it follows that
\begin{equation}
P_{v_{1}}(\theta) = P_{v_{2}}(\theta) \ \ \ {\text{or}} \ \ \ P_{v_{1}}(\theta) \cap P_{v_{2}}(\theta) = \emptyset.
\end{equation}
If $P_{v_{1}}(\theta) \cap P_{v_{2}}(\theta) = \emptyset$, there exists $h \in G$, such that $P_{v_{2}}(\theta) = P_{v_{1}h}(\theta)$. In fact, considering the parallel transport $\tau \colon \pi^{-1}(\pi(v_{2})) \to \pi^{-1}(\pi(v_{1}))$, along some curve joining $\pi(v_{2})$ and $\pi(v_{1})$, we have $v_{2} \sim \tau(v_{2}) = v_{1}h$, for some $h \in G$. Moreover, since 
\begin{equation}
R_{h} \colon (P_{v_{1}}(\theta),g_{P}|_{P_{v_{1}}(\theta)}) \to (P_{v_{2}}(\theta),g_{P}|_{P_{v_{2}}(\theta)}),
\end{equation}
maps each horizontal curve into horizontal curve, it follows that $R_{h}$ defines an isometry, notice that $R_{h}^{\ast}g_{P} = g_{P}$, $\forall h \in G$. In particular, since
\begin{equation}
\ker(\theta_{v_{2}}) = (R_{h})_{\ast}\ker(\theta_{v_{1}}),
\end{equation}
denoting  $\mathcal{H}_{j} := \ker(\theta_{v_{2}})|_{P_{v_{j}}(\theta)}$, $j = 1,2$, it follows that 
\begin{equation}
{\rm{diam}}(P_{v_{2}}(\theta),{\rm{dist}}_{\mathcal{H}_{2}}) = {\rm{diam}}(P_{v_{1}}(\theta),{\rm{dist}}_{\mathcal{H}_{1}}),
\end{equation}
i.e., $ {\rm{diam}}(P_{b}(\theta),{\rm{dist}}_{\mathcal{H}})$ in Eq. (\ref{fundinequality}) does not depend on $b \in P$. Therefore, by setting
\begin{equation}
\varkappa_{0}  := {\rm{diam}}(P_{u}(\theta),{\rm{dist}}_{\mathcal{H}}),
\end{equation}
where ${\rm{dist}}_{\mathcal{H}}$ is the Carnot-Carath\'{e}odory metric induced by the bundle type sub-Riemannnian structure $(\ker(\theta)|_{P_{u}(\theta)},g_{P}|_{P_{u}(\theta)})$ on $P_{u}(\theta)$, we conclude that 
\begin{equation}
d_{g_{t}}(a,b) - d_{g_{T}}(F(a),F(b)) \leq \varkappa_{0}f(t),
\end{equation}
for every $a,b \in P$. Moreover, since $(P,g_{t})$ is a complete Riemannian manifold, we have a minimizing geodesic $\gamma \colon [0,1] \to P$, connecting $a$ and $b$, thus
\begin{equation}
d_{g_{t}}(a,b) = \int_{0}^{1}||\dot{\gamma}(s)||_{g_{t}}{\rm{d}}s  \geq \int_{0}^{1}||F_{\ast}\dot{\gamma}(s)||_{g_{T}}{\rm{d}}s \geq d_{g_{T}}(F(a),F(b)).
\end{equation}
Hence, we obtain
\begin{equation}
\label{fundamentaldist}
|d_{g_{t}}(a,b) - d_{g_{T}}(F(a),F(b))| \leq \varkappa_{0}f(t),
\end{equation}
for every $a,b \in P$. Considering now the correspondence
\begin{equation}
\mathscr{R}_{F} = \Big \{ (u,F(u)) \in P \times (G/{\rm{Hol}}_{u}(\theta)) \ \ \Big | \ \ u \in P\Big \},
\end{equation}
it follows from Eq. (\ref{fundamentaldist}) that 
\begin{equation}
{\rm{dis}}(\mathscr{R}_{F}) = \sup \Big \{ | d_{g_{t}}(a,b) - d_{g_{T}}(F(a),F(b))| \ \big | \ (a,F(a)), (b,F(b)) \in \mathscr{R}_{F}\Big \} \leq \varkappa_{0}f(t).
\end{equation}
From above, we obtain 
\begin{equation}
d_{GH}\big ((P,d_{g_{t}}),(G/{\rm{Hol}}_{u}(\theta),d_{g_{T}})\big)  = \frac{1}{2} \inf \Big \{ {\rm{dis}}(\mathscr{R}) \ \ \Big | \ \ \mathscr{R} \in \mathcal{R}(P,G/{\rm{Hol}}_{u}(\theta)) \Big \} \leq \frac{\varkappa_{0}f(t)}{2},
\end{equation}
for every $t \in [0,T)$, which concludes the proof.
\end{proof}
From above theorem, we immediately have the following corollaries.
\begin{corollary}
Under the hypotheses of the last theorem, if $\lim_{t \to T}f(t) = 0$, then
\begin{equation}
\lim_{t \to T} d_{GH}\big ((P,d_{g_{t}}),(G/{\rm{Hol}}_{u}(\theta),d_{g_{T}})\big)  = 0.
\end{equation}
\end{corollary}

\begin{corollary}
Under the hypotheses of the last theorem, 
\begin{equation}
d_{GH}\big ((P,d_{g_{P}}),(G/{\rm{Hol}}_{u}(\theta),d_{g_{T}})\big) \leq \frac{{\rm{diam}}\big (P_{u}(\theta),{\rm{dist}}_{\mathcal{H}} \big )}{2},
\end{equation}
where $P_{u}(\theta)$ is the holonomy bundle through $u \in P$ and ${\rm{dist}}_{\mathcal{H}}$ is the Carnot-Carath\'{e}odory metric induced by the bundle type sub-Riemannnian structure $(\ker(\theta),g_{P})$ restricted to $P_{u}(\theta)$.
\end{corollary}
\begin{corollary}
In the setting of the last corollary, if $g_{P}(t) := \pi^{\ast}g_{M}(t) + \langle \theta,\theta \rangle$, $t \in [0,T)$, for some family of metrics $\{g_{M}(t)\}_{t \in [0,T)}$ on the base manifold $M$, then
\begin{equation}
d_{GH}\big ((P,d_{g_{P}(t)}),(G/{\rm{Hol}}_{u}(\theta),d_{g_{T}})\big) \leq \frac{{\rm{diam}}\big (P_{u}(\theta),{\rm{dist}}_{\mathcal{H}}(t) \big )}{2},
\end{equation}
where $P_{u}(\theta)$ is the holonomy bundle through $u \in P$ and ${\rm{dist}}_{\mathcal{H}}(t)$ is the Carnot-Carath\'{e}odory metric induced by the bundle type sub-Riemannnian structure $(\ker(\theta),g_{P}(t))$ restricted to $P_{u}(\theta)$.
\end{corollary}

\section{Proof of Theorem B}

In order to prove Theorem \ref{thmB}, we recall the following well-known result (e.g. \cite{nomizu1955reduction, nomizu1956theoreme}).

\begin{theorem}
\label{Nomizu}
Let $\pi \colon P \to M$ be a principal $G$-bundle and let $H \subset G$ be a connected Lie subgroup. Suppose that $\dim(M) \geq 2$. Then there exists a principal connection on $P$ whose holonomy group is $H$ if, and only if, the structure group $G$ can be reduced to $H$. 
\end{theorem}

From the above theorem and the ideas presented in the previous section, we obtain the following result (Theorem \ref{thmB}).

\begin{theorem}
Let $\pi \colon P \to M$ be a principal $G$-bundle, such that $M$ and $G$ are both compact and connected, and $\dim(M) \geq 2$. Consider the subset $\mathcal{M}(P) \subset (\mathcal{M},d_{GH})$ defined by 
\begin{equation}
\mathcal{M}(P) := \big \{ (P,d_{g}) \ | \ d_{g} \ \text{is the distance induced by} \ g \in {\rm{Sym}}_{+}^{2}(T^{\ast}P) \big \}.
\end{equation}
If $\pi \colon P \to M$ is reducible to a closed connected subgroup $H \subset G$, then $(G/H,d_{g_{T}}) \in \overline{\mathcal{M}(P)}^{GH}$.
\end{theorem}

\begin{proof}
Suppose that the structure group of $\pi \colon P \to M$ can be reduced to a closed connected subgroup $H \subset G$. From Cartan's theorem, we have that $H \subset G$ is a closed connected Lie subgroup. Applying Theorem \ref{Nomizu}, we have a connection $\theta \in \Omega^{1}(P,\mathfrak{g})$, such that ${\rm{Hol}}_{u}(\theta) = H$, for some $u \in P$. Given $n \in \mathbb{N}$, by setting $f_{n} \equiv \frac{1}{n}$, we define
\begin{equation}
g_{n} := {\textstyle{\frac{1}{n}}}\pi^{\ast}g_{M} + \langle \theta,\theta \rangle  \ \ \text{and} \ \ g_{P} := \pi^{\ast}g_{M} + \langle \theta,\theta \rangle,
\end{equation}
where $g_{M}$ is some Riemannian metric on $M$ and $\langle -,-\rangle $ is some bi-invariant metric on $G$. From Theorem \ref{holonomy-GH}, it follows that   
\begin{equation}
d_{GH}\big ((P,d_{g_{n}}),(G/H,d_{g_{T}})\big) \leq \frac{{\rm{diam}}\big (P_{u}(\theta),{\rm{dist}}_{\mathcal{H}} \big )}{n},
\end{equation}
for every $n \in \mathbbm{N}$, where $P_{u}(\theta)$ is the holonomy bundle and ${\rm{dist}}_{\mathcal{H}}$ is the Carnot-Carath\'{e}odory metric induced by the bundle type sub-Riemannnian structure $(\ker(\theta),g_{P})$ restricted to $P_{u}(\theta)$. Therefore, we obtain a sequence $(P,d_{g_{n}}) \in \mathcal{M}(P)$, $n \in \mathbbm{N}$, such that 
\begin{equation}
\lim_{n\to \infty}d_{GH}\big ((P,d_{g_{n}}),(G/H,d_{g_{T}})\big) = 0.
\end{equation}
From above, we conclude $(G/H,d_{g_{T}}) \in \overline{\mathcal{M}(P)}^{GH}$.
\end{proof}

\subsection*{Data Availability} Data sharing not applicable to this article as no datasets were generated or analysed during the current study.
\subsection*{Conflict of interest statement} The authors declare that there is no conflict of interest.

\bibliographystyle{alpha}
\bibliography{ref.bib}

\end{document}